\documentclass{article}
\usepackage{amsmath}
\usepackage{amssymb}

\def\ec{{\rm ec}\hskip0.02cm}
\def\ci{{\rm CI}\hskip0.02cm}

\newtheorem{theorem}{Theorem}
\newtheorem{definition}{Definition}

\newtheorem{proposition}{Proposition}

\newenvironment{remark}[1]{\medskip\par\noindent{\bf #1.}\rm}
{\medskip\par\noindent}

\newenvironment{proof}[1]{\medskip\par\noindent{\sc #1.\ }}
{~\rule{0.5em}{0.5em}\medskip\par}
\begin{document}

\title{\LARGE Minimum congestion spanning trees in planar graphs}

\author{M.~I.~Ostrovskii\footnote{Supported by St. John's University Summer 2009 Support
of Research Program}\\
Department of Mathematics and Computer Science\\
St. John's University\\
8000 Utopia Parkway\\
Queens, NY 11439, USA\\
Fax:(718)-990-1650\\
Phone: (718)-990-2469\\
e-mail: {\tt ostrovsm@stjohns.edu}}

\date{\today}
\maketitle

\noindent{\bf Abstract:} The main purpose of the paper is to
develop an approach to evaluation or estimation of the spanning
tree congestion of planar graphs. This approach is used to
evaluate the spanning tree congestion of triangular grids.
\medskip


\begin{large}

%\tableofcontents

\section{Introduction}

Let $G$ be a graph and let $T$ be a spanning tree in $G$ (saying
this we mean that $T$ is a subgraph of $G$). We follow the
terminology and notation of \cite{CH91}. For each edge $e$ of $T$
let $A_e$ and $B_e$ be the vertex sets of the components of $T-e$.
By $e_G(A_e,B_e)$ we denote the number of edges in $G$ with one
end vertex in $A_e$ and the other end vertex in $B_e$. We define
the {\it edge congestion} of $G$ in $T$ by
$$
\ec(G:T)=\max_{e\in E_T} e_G(A_e,B_e).
$$
The number $e_G(A_e,B_e)$ is called the {\it congestion} in $e$.
The name comes from the following analogy. Imagine that edges of
$G$ are roads, and edges of $T$ are those roads which are cleaned
from snow after snowstorms. If we assume that each edge in $G$
bears the same amount of traffic, and that after a snowstorm each
driver takes the corresponding (unique) detour in $T$, then
$\ec(G:T)$ describes the traffic congestion at the most congested
road of $T$. It is clear that for applications it is interesting
to find a spanning tree which minimizes the congestion.

We define {\it the spanning tree congestion} of $G$ by
\begin{equation}\label{E:s}
s(G)=\min\{\ec (G:T):~T\hbox{ is a spanning tree of }G\}.
\end{equation}
Each spanning tree $T$ in $G$ satisfying $\ec(G:T)=s(G)$ is called
a {\it minimum congestion spanning tree}. The parameters
$\ec(G:T)$ and $s(G)$ were introduced and studied in \cite{Ost04}.
This study was continued in
\cite{Car05,CCV07,CO09,Cox07,Hru08,Hun07,KOY09,LRR09,Tan07}, where
many interesting results were obtained.
\medskip

The spanning tree congestion is of interest in the study of
Banach-space-theoretical properties of Sobolev spaces on graphs,
see \cite{Ost05}. Many known results and algorithms related to
spanning trees are collected in the monograph \cite{WC04}, but
this monograph does not contain any results on the spanning tree
congestion. Many related parameters have been introduced in the
literature, see \cite{BCHRS00,KRY94} and references therein, the
paper \cite{KRY94} introduced parameters which are more general
than the spanning tree congestion. \medskip

One of the interesting problems about the spanning tree congestion
is to evaluate it for some natural families of graphs. The purpose
of this paper is to develop techniques which can be used to
evaluate or estimate the spanning tree congestion of planar
graphs. The techniques uses duality for planar graphs which goes
back to Poincar\`e and Whitney (see \cite[Section 8.8.2]{Ros99}
and \cite{Whi32,Whi33}) and the notion of a dual tree which is
implicitly present in the work of Whitney (see \cite[Problems 5.23
and 5.36]{Lov79}). Dual trees were introduced to this area by
Hruska \cite{Hru08} who used them to evaluate the spanning tree
congestion for rectangular planar grids.
\medskip

In conclusion we would like to mention that another techniques
used to estimate the spanning tree congestion is based on the
notion of a centroid of a tree (see \cite[p.~46]{WC04} or
\cite{Ost04} for the definition) and edge-isoperimetric
inequalities. This techniques was initiated in \cite{Ost04} and
developed in \cite{CO09} and \cite{KOY09}. It would be interesting
to obtain the results for triangular grid (Theorem \ref{T:triang})
using isoperimetry.

\section{Dual graphs and spanning tree congestion estimates}
\label{S:dual}

Let $G$ be a connected plane graph, that is, a planar graph with a
fixed drawing in the plane.

\begin{definition}\label{D:dualTree}
{\rm The {\it dual graph} $G^*$ of $G$ is defined as the graph
whose vertices are faces of $G$, including the exterior
(unbounded) face, and whose edges are in a bijective
correspondence with edges of $G$. The edge $e^*\in E(G^*)$
corresponding to $e\in E(G)$ joins the faces which are on
different sides of the edge $e$.
\smallskip

Let $T$ be a spanning tree of $G$. The {\it dual tree $T^\sharp$}
is defined as a spanning subgraph of $G^*$ whose edge set
$E(T^\sharp)$ is determined by the condition: $e^*\in E(T^\sharp)$
if and only if $e\notin E(T)$.}
\end{definition}

\begin{remark}{Note} The graph $G^*$ does not have to be a simple graph even when
$G$ is simple. It is easy to verify that $T^\sharp$ is a spanning
tree in $G^*$ (see \cite[Solution of Problem 5.23]{Lov79}). See
\cite[Section 5.6]{CH91} and \cite{Whi32,Whi33,Lov79} for
information about dual graphs.
\end{remark}

\begin{definition} {\rm
Let $e\in E(G)$. We say that $e$ is an {\it outer edge} if it is
an edge which occurs in the boundary of the exterior face and one
of the interior faces. For each outer edge $e$ and each bounded
face $F$ of $G$ define the {\it index} $i(F,e)$ as the length of a
shortest path in $G^*$ which joins the exterior face $O$ with $F$
and satisfies the additional condition: its first edge is $e^*$.}
\end{definition}

\begin{definition}\label{D:CenTail}
{\rm A {\it center-tail} system $\mathcal{S}$ in the dual graph
$G^*$ of a plane graph $G$ consists of

\begin{itemize}
\item[(1)] A set $C$ of vertices of $G^*$ spanning a connected
subgraph of $G^*$, the set $C$ is called a {\it center}.

\item[(2)] A set of paths in $G^*$ joining some vertices of the
center with the exterior face $O$. Each such path is called a {\it
tail}. The {\it tip of a tail} is the last vertex of the
corresponding path before it reaches the exterior face.

\item[(3)] An assignment of {\it opposite tails} for outer edges
of $G$. This means: For each outer edge $e$ of the graph $G$ one
of the tails is assigned to be the {\it opposite tail} of $e$, it
is denoted $N(e)$ and its tip is denoted by $t(e)$.
\end{itemize}
}\end{definition}

See Section \ref{S:triang} for examples of center-tail systems.
\medskip

The result below is true for an arbitrary system $\mathcal{S}$
satisfying the relations described above, but to be useful for
estimates of the spanning tree congestion, a center should consist
of vertices which are far from the exterior face in $G^*$ and
opposite tails should be tails which are in some natural metric
sense go in the directions which are opposite to the corresponding
edges.

\begin{definition}\label{D:CongInd}
{\rm The {\it congestion indicator} $\ci(\mathcal{S})$  of a
center-tail system $\mathcal{S}$ is defined as the minimum of the
following three numbers:

\begin{itemize}
\item[(1)] $\min_{F,H,f,h}(i(F,f)+i(H,h)+1)$, where the minimum is
taken over all pairs $F,H$ of adjacent vertices in the center $C$
and over all pairs $f,h$ of outer edges with $f\ne h$. In the
cases where the center contains just one face we assume that this
minimum is $\infty$.

\item[(2)] $\min_e i(t(e),e)+1$, where the minimum is taken over
all outer edges of $G$.

\item[(3)] $\min_e\min_{F\in N(e)}\min_{\tilde e\ne e}
(i(F,e)+i(\widetilde{F},\tilde{e})+1)$, where the first minimum is
taken over all outer edges of $G$;  the second minimum is over
vertices $F$ from the path $N(e)$, $\widetilde{F}$ is the vertex
in $N(e)$ which follows immediately after $F$ if one moves along
$N(e)$ from $F$ to $t(e)$; and the third minimum is over all outer
edges different from $e$,
\end{itemize}
}
\end{definition}

\begin{theorem}\label{T:ICvsS} Let $\mathcal{S}$ be any center-tail system in
a connected planar graph $G$. Then $s(G)\ge \ci(\mathcal{S})$.
\end{theorem}

\begin{proof}{Proof} Let $T$ be a spanning tree in $G$ and
$T^\sharp$ be its dual tree. We split the set of interior faces of
$G$ into branches corresponding to outer edges (many of the
branches can be empty): the {\it branch} corresponding to an outer
edge $e$ is the set of faces  which are separated from the
exterior face $O$ if we delete $e^*$ from $T^\sharp$, we assume
that the branch is empty if $e^*$ is not an edge of $T^\sharp$.
The edge $e$ (corresponding to $e^*$) is called the {\it entrance}
of the branch.
\medskip

{\bf Observation 1.} If faces of the center $C$ belong to
different branches, then $\ec(G:T)\ge$ the minimum in item (1) of
Definition \ref{D:CongInd}.
\medskip

In fact, let $F$ and $H$ be faces which are adjacent in $G^*$ and
belong to two different branches with entrances at $f$ and $h$,
respectively. Let $g$ be an edge which occurs in the boundaries of
the faces $F$ and $H$. It is clear that $g\in E(T)$ (otherwise $T$
would be disconnected). It suffices to show that $e_G(A_g,B_g)\ge
i(F,f)+i(H,h)+1$.
\medskip

Let $f^*_1,\dots,f^*_k$ be the $OF$-path in $T^\sharp$ and
$h_1^*,\dots,h_m^*$ be the $OH$-path in $T^\sharp$. It is clear
that $k\ge i(F,f)$ and $m\ge i(H,h)$. To complete the proof we
show that $g$ is used in detours for
$f_1,\dots,f_k,h_1,\dots,h_m$, and itself. In fact, the edges
$f_1^*,\dots,f^*_k,g^*,h^*_m,\dots,h^*_1$ form a cycle in $G^*$.
Hence the edges $f_1,\dots, f_k, g, h_m,\dots,h_1$ form a cut in
$G$, and $g$ is the only edge in $T$ connecting the vertex sets
separated by the cut. This completes our proof of Observation 1.
\medskip

{\bf Observation 2.} Suppose that all faces of the center belong
to the same branch with entrance $e$. Then:
\medskip

(a) If all faces from the tail $N(e)$ also belong to the same
branch (with entrance $e$), then $\ec(G:T)\ge$ the minimum in item
(2) of Definition \ref{D:CongInd}.
\medskip

(b) If some faces from the tail $N(e)$ belong to another branch,
then $\ec(G:T)\ge$ the minimum in item (3) of Definition
\ref{D:CongInd}.
\medskip

In fact, in the case (a) let $g$ be an edge which occurs in the
boundaries of the the tip $t(e)$ and the outer face $O$. As in
Observation 1 we get that $g\in E(T)$ (otherwise $T$ would be
disconnected) and that $e_G(A_g,B_g)\ge i(t(e),e)+1$.
\medskip

In the case (b) let $\widetilde{F}$ be the first face on the path
$N(e)$ (we assume that the path starts at a vertex of the center)
which belongs to a different branch and let $F$ be the previous
edge of the path. Let $g\in E(G)$ be the edge corresponding to the
edge of $g^*\in E(G^*)$ joining $F$ and $\widetilde{F}$. Let
$\tilde e$ be the entrance of the branch to which $\widetilde{F}$
belongs. As in the previous observations we show that
$e_G(A_g,B_g)\ge i(F,e)+i(\widetilde{F},\tilde{e})+1$; and we are
done.\medskip

It is clear that together Observations 1 and 2 imply the statement
of the theorem.
\end{proof}

In our estimates of $s(G)$ from above we use the following
definition.

\begin{definition} {\rm The  {\it absolute
index} $i(F)$ of a face $F$ is defined as $\min_ei(F,e)$, where
the minimum is over all outer edges.}
\end{definition}

\begin{proposition}\label{P:above} For each connected planar graph $G$ we have
\begin{equation}\label{E:above}s(G)\le\max (i(F)+i(\widetilde{F}))+1,\end{equation} where the maximum is
over all pairs $F,\widetilde{F}$ of faces which have a common edge
in their boundaries.
\end{proposition}

\begin{proof}{Proof} We let $T^\sharp$ be a so-called breadth-first-search (BFS) tree in
$G^*$ rooted at the external face $O$. See \cite[Section
9.2.1]{Ros99} for a definition of a  breadth-first-search tree.
The definition in \cite{Ros99} explains the name. We need only the
following defining property of a BFS tree in a connected graph
$H$: it is a spanning tree in $H$ for which the distance between
any vertex and the root in the tree is the same as in $H$. It is
easy to see that BFS trees exist in an arbitrary connected graph.
\medskip

So let $T^\sharp$ be a rooted at $O$ BFS tree in $G^*$ and
$E(T^\sharp)$ be the edge set of $T^\sharp$. We delete from $E(G)$
the set $\{e: e^*\in E(T^\sharp)\}$. It is easy to check that we
get a spanning tree (see \cite[Solution of Problem 5.23]{Lov79}
for detailed explanation). We denote it by $T$ because $T^\sharp$
is its dual tree.
\medskip

Consider any edge $f\in E(T)$. Suppose that it occurs in the
boundaries of faces of $F_1$ and $F_2$.
\medskip

Observe that the number of edges detours for which use $f$ is
equal to the number of edges in the cycle contained in
$T^\sharp\cup\{f^*\}$. The length of the cycle is $\le
i(F_1)+i(F_2)+1$ because the cycle is a part of the closed walk
which starts at $O$, uses an $OF_1$-path in $T^\sharp$, then
$f^*$, and then an $F_2O$-path in $T^\sharp$.
\end{proof}

\begin{remark}{Remark} The proof of Proposition \ref{P:above} explains why we do not have the equality
in \eqref{E:above}: in some cases the cycles contained in
$T^\sharp\cup\{f^*\}$, where $f^*$ is an edge joining faces $F$
and $\widetilde{F}$ maximizing $i(F)+i(\widetilde{F})$ do not pass
through $O$. To illustrate this remark we consider the following
planar graph $H$: It is obtained if we consider $n$ concentric
circles and $k$ radial line segments, $n\gg k$. Each intersection
of a circle and a line segment is regarded as a vertex. (See
Figure 1, where $n=3$, $k=4$, and ``circles'' are sketched as
squares. We do not have $n\gg k$ in this picture, but it shows how
we construct the spanning tree (drawn using ``fat'' edges), also
it shows values of absolute indices of different faces.)
\medskip

\begin{center}
%Picture 2
\setlength{\unitlength}{2mm}
\begin{picture}(58,27)
%\linethickness{0.01pt}
%\graphpaper[1](0,0)(30,12)
%\thinlines

\put(19,12){1} \put(23,12){2} \put(27,12){3} \put(31,12){3}
\put(35,12){2} \put(39,12){1}

\put(29,23){1} \put(29,19){2} \put(29,15){3} \put(29,10){3}
\put(29,7){2} \put(29,3){1}

\linethickness{1.5pt}\qbezier(29,13)(29,13)(41,25)
\linethickness{0.3pt}\qbezier(29,13)(29,13)(17,1)
%\linethickness{0.3pt}\qbezier(17,26)(17,26)(41,2)

\linethickness{0.3pt}\qbezier(17,25)(17,25)(41,1)

\linethickness{1.5pt}\qbezier(17,25)(17,25)(17,1)
\linethickness{1.5pt}\qbezier(21,21)(21,21)(21,5)
\linethickness{1.5pt}\qbezier(25,17)(25,17)(25,9)

\linethickness{1.5pt}\qbezier(41,25)(41,25)(41,1)
\linethickness{1.5pt}\qbezier(37,21)(37,21)(37,5)
\linethickness{1.5pt}\qbezier(33,17)(33,17)(33,9)

\linethickness{1.5pt}\qbezier(17,25)(17,25)(41,25)
\linethickness{1.5pt}\qbezier(21,21)(21,21)(37,21)
\linethickness{1.5pt}\qbezier(25,17)(25,17)(33,17)

\linethickness{0.3pt}\qbezier(17,1)(17,1)(41,1)
\linethickness{0.3pt}\qbezier(21,5)(21,5)(37,5)
\linethickness{0.3pt}\qbezier(25,9)(25,9)(33,9)

\end{picture}
\bigskip

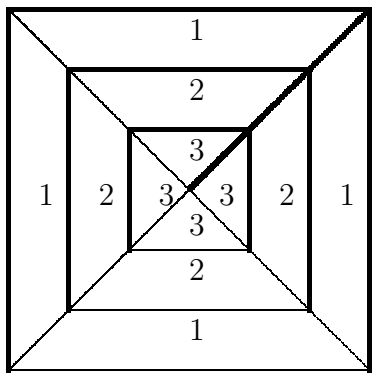
\centerline{Figure 1}
\end{center}

For such graph the absolute indices $i(F)$ of faces $F$ contained
in the smallest circle are equal to $n$. On the other hand, it is
easy to check that the spanning tree $T$ in $H$ consisting of all
edges from one of the line segments and all edges from circles
with one edge per circle removed satisfies $\ec(G:T)\le 2k$
(actually, if we remove edges from circles in an optimal way, it
will satisfy $\ec(G:T)\le k+2$, see Figure 1).
\end{remark}

\section{Triangular grids}\label{S:triang}

Now we are going to use center-tail systems to find the spanning
tree congestion for triangular grids $\{T_k\}_{k=2}^n$. The graph
$T_k$ is defined as the graph which we obtain if we divide each
side of a triangle into $k-1$ equal pieces and join the
corresponding subdivision points of different sides of the
triangle. To make this definition clear we sketch $T_2$, $T_3$,
and $T_4$ (see Figure 2). In these graphs all intersections of
line segments are regarded as vertices, and there are no other
vertices.

\begin{center}
\setlength{\unitlength}{2mm}
\begin{picture}(58,27)

\linethickness{0.3pt}\qbezier(7,1)(7,1)(14.6,1)
\linethickness{0.3pt}\qbezier(7,1)(7,1)(10.8,7.7)
\linethickness{0.3pt}\qbezier(10.8,7.7)(10.8,7.7)(14.6,1)

\linethickness{0.3pt}\qbezier(17,1)(17,1)(32.2,1)
\linethickness{0.3pt}\qbezier(17,1)(17,1)(24.6,14.4)
\linethickness{0.3pt}\qbezier(24.6,14.4)(24.6,14.4)(32.2,1)

\linethickness{0.3pt}\qbezier(20.8,7.7)(20.8,7.7)(28.4,7.7)
\linethickness{0.3pt}\qbezier(20.8,7.7)(20.8,7.7)(24.6,1)
\linethickness{0.3pt}\qbezier(24.6,1)(24.6,1)(28.4,7.7)

\linethickness{0.3pt}\qbezier(35,1)(35,1)(57.8,1)
\linethickness{0.3pt}\qbezier(35,1)(35,1)(46.4,21.1)
\linethickness{0.3pt}\qbezier(46.4,21.1)(46.4,21.1)(57.8,1)

\linethickness{0.3pt}\qbezier(38.8,7.7)(38.8,7.7)(54,7.7)
\linethickness{0.3pt}\qbezier(42.6,14.4)(42.6,14.4)(50.2,14.4)

\linethickness{0.3pt}\qbezier(38.8,7.7)(38.8,7.7)(42.6,1)
\linethickness{0.3pt}\qbezier(42.6,14.4)(42.6,14.4)(50.2,1)

\linethickness{0.3pt}\qbezier(54,7.7)(54,7.7)(50.2,1)
\linethickness{0.3pt}\qbezier(50.2,14.4)(50.2,14.4)(42.6,1)

\end{picture}
\bigskip

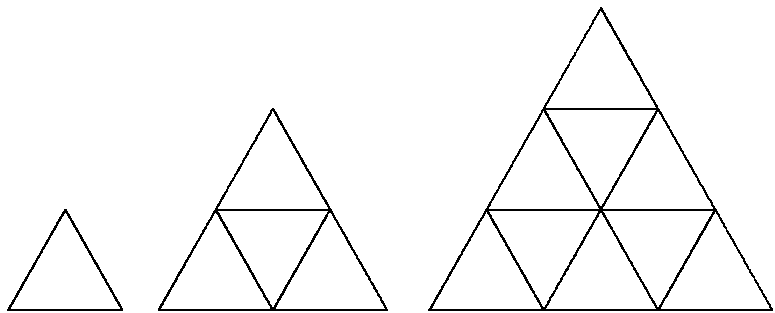
\centerline{Figure 2}
\end{center}

\begin{theorem}\label{T:triang} $s(T_{3n})=4n$, $s(T_{3n+1})=4n$,
$s(T_{3n+2})=4n+2$, $n=0,1,2,\dots$.
\end{theorem}

\begin{proof}{Proof} To estimate the spanning tree congestion from
below we use center-tail systems. Our descriptions of center-tail
systems $\mathcal{S}_n$ for $T_n$ are somewhat different in the
cases when $n=3k$, $n=3k+1$, and $n=3k+2$.
\medskip

We shall give a detailed argument for $n=5,6,7$ and use the
induction to derive the formula from the statement of the theorem.
\medskip

The case $n=5$. The center-tail system $\mathcal{S}_5$ is
described in the following way. The triangle containing the letter
$C$ (see Figure 3) is the only element of the center. There are
three tails, shown in Figure 3 using ``fat'' lines; we do not show
edges joining tips of tails and $O$. The tail going in the
upward-right direction is assigned to be the opposite tail for all
outer edges contained in the bottom side of the triangle.
Assignment of the opposite tails to edges from other sides of the
triangle is made in order to make the assignment rotationally
invariant for angles of $120^\circ$ and $240^\circ$.

\begin{center}
\setlength{\unitlength}{2mm}
\begin{picture}(58,29)

\linethickness{0.3pt}\qbezier(13.8,1)(13.8,1)(44.2,1)
\linethickness{0.3pt}\qbezier(13.8,1)(13.8,1)(29,27.8)
\linethickness{0.3pt}\qbezier(29,27.8)(29,27.8)(44.2,1)

\linethickness{0.3pt}\qbezier(17.6,7.7)(17.6,7.7)(40.4,7.7)
\linethickness{0.3pt}\qbezier(21.4,14.4)(21.4,14.4)(36.6,14.4)
\linethickness{0.3pt}\qbezier(25.2,21.1)(25.2,21.1)(32.8,21.1)

\linethickness{0.3pt}\qbezier(17.6,7.7)(17.6,7.7)(21.4,1)
\linethickness{0.3pt}\qbezier(21.4,14.4)(21.4,14.4)(29,1)
\linethickness{0.3pt}\qbezier(25.2,21.1)(25.2,21.1)(36.6,1)

\linethickness{0.3pt}\qbezier(21.4,1)(21.4,1)(32.8,21.1)
\linethickness{0.3pt}\qbezier(29,1)(29,1)(36.6,14.4)
\linethickness{0.3pt}\qbezier(36.6,1)(36.6,1)(40.4,7.7)

\put(27,8.4){$C$}

\linethickness{1.5pt}\qbezier(21.4,9.9)(21.4,9.9)(25.2,12.1)
\linethickness{1.5pt}\qbezier(25.2,12.1)(25.2,12.1)(29,9.9)
\linethickness{1.5pt}\qbezier(29,9.9)(29,9.9)(29,5.4)

\linethickness{1.5pt}\qbezier(29,5.4)(29,5.4)(32.8,3.2)
\linethickness{1.5pt}\qbezier(29,9.9)(29,9.9)(32.8,12.1)
\linethickness{1.5pt}\qbezier(32.8,12.1)(32.8,12.1)(32.8,16.6)

\end{picture}
\bigskip

\centerline{Figure 3}
\end{center}

Now we evaluate for $\mathcal{S}_5$ all of the minima from the
definition of the $\ci$. The first minimum is $\infty$ by the
definition because the center contains just one face.
\medskip

Because of the symmetry in the second and the third minimum it
suffices to consider the minima over the edges of the bottom side
only. To do this it is convenient to introduce $i_{\rm
bot}(F)=\min i(F,e)$, where $F$ is a face of the triangle and the
minimum is over $e$ from the bottom side of the triangle. The
values of $i_{\rm bot}$ are shown in Figure 4. It is clear from
Figure 4 that for each $e$ from the bottom side of the triangle we
have $i(t(e),e)+1\ge 6$.

\begin{center}
\setlength{\unitlength}{2mm}
\begin{picture}(58,29)

\linethickness{0.3pt}\qbezier(13.8,1)(13.8,1)(44.2,1)
\linethickness{0.3pt}\qbezier(13.8,1)(13.8,1)(29,27.8)
\linethickness{0.3pt}\qbezier(29,27.8)(29,27.8)(44.2,1)

\linethickness{0.3pt}\qbezier(17.6,7.7)(17.6,7.7)(40.4,7.7)
\linethickness{0.3pt}\qbezier(21.4,14.4)(21.4,14.4)(36.6,14.4)
\linethickness{0.3pt}\qbezier(25.2,21.1)(25.2,21.1)(32.8,21.1)

\linethickness{0.3pt}\qbezier(17.6,7.7)(17.6,7.7)(21.4,1)
\linethickness{0.3pt}\qbezier(21.4,14.4)(21.4,14.4)(29,1)
\linethickness{0.3pt}\qbezier(25.2,21.1)(25.2,21.1)(36.6,1)

\linethickness{0.3pt}\qbezier(21.4,1)(21.4,1)(32.8,21.1)
\linethickness{0.3pt}\qbezier(29,1)(29,1)(36.6,14.4)
\linethickness{0.3pt}\qbezier(36.6,1)(36.6,1)(40.4,7.7)

\put(17.6,3.2){1} \put(25.2,3.2){1} \put(32.8,3.2){1}
\put(40.4,3.2){1}

\put(25.2,12.1){4} \put(32.8,12.1){4}

\put(25.2,16.6){5} \put(32.8,16.6){5}

\put(21.4,5.4){2} \put(29,5.4){2} \put(36.6,5.4){2}

\put(21.4,9.9){3} \put(29,9.9){3} \put(36.6,9.9){3}

\put(29,18.8){6}

\put(29,23.3){7}
\end{picture}
\bigskip

\centerline{Figure 4}
\end{center}

The values of the absolute index $i(F)$ for faces of $T_5$ are
shown in Figure 5. Comparing Figures 4 and 5 we see that for each
face $F$ from $N(e)$ for $e$ from the bottom side, and the
following face $\widetilde{F}$ the sum $i(F,e)+\min_{\tilde e\ne
e} i(\widetilde{F},\tilde e)+1$ is at least $6$.

\begin{center}
\setlength{\unitlength}{2mm}
\begin{picture}(58,29)

\linethickness{0.3pt}\qbezier(13.8,1)(13.8,1)(44.2,1)
\linethickness{0.3pt}\qbezier(13.8,1)(13.8,1)(29,27.8)
\linethickness{0.3pt}\qbezier(29,27.8)(29,27.8)(44.2,1)

\linethickness{0.3pt}\qbezier(17.6,7.7)(17.6,7.7)(40.4,7.7)
\linethickness{0.3pt}\qbezier(21.4,14.4)(21.4,14.4)(36.6,14.4)
\linethickness{0.3pt}\qbezier(25.2,21.1)(25.2,21.1)(32.8,21.1)

\linethickness{0.3pt}\qbezier(17.6,7.7)(17.6,7.7)(21.4,1)
\linethickness{0.3pt}\qbezier(21.4,14.4)(21.4,14.4)(29,1)
\linethickness{0.3pt}\qbezier(25.2,21.1)(25.2,21.1)(36.6,1)

\linethickness{0.3pt}\qbezier(21.4,1)(21.4,1)(32.8,21.1)
\linethickness{0.3pt}\qbezier(29,1)(29,1)(36.6,14.4)
\linethickness{0.3pt}\qbezier(36.6,1)(36.6,1)(40.4,7.7)

\put(17.6,3.2){1} \put(25.2,3.2){1} \put(32.8,3.2){1}
\put(40.4,3.2){1}

\put(25.2,12.1){2} \put(32.8,12.1){2}

\put(25.2,16.6){1} \put(32.8,16.6){1}

\put(21.4,5.4){2} \put(29,5.4){2} \put(36.6,5.4){2}

\put(21.4,9.9){1} \put(29,9.9){3} \put(36.6,9.9){1}

\put(29,18.8){2}

\put(29,23.3){1}
\end{picture}
\bigskip

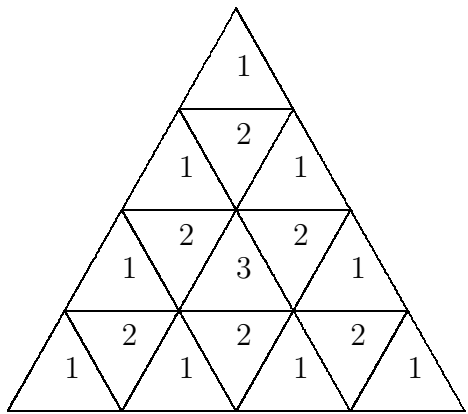
\centerline{Figure 5}
\end{center}

\medskip

By Theorem \ref{T:ICvsS} we get $s(T_5)\ge 6$.
\medskip

Applying Proposition \ref{P:above} to the values of $i(F)$ in
Figure 5 we get $s(T_5)\le 6$.
\medskip

Observe that if we add one row on each side of $T_5$ we get $T_8$,
the index of each triangle from $T_5$ increases by $2$. If we
construct $\mathcal{S}_8$ in a similar way (that is, letting $C$
to be the central face and extending each of the tails by two
edges), we get $\ci(\mathcal{S}_8)=10$. Applying Proposition
\ref{P:above} we get $s(T_8)=10$.\medskip

It is easy to see that the same pattern repeats. Each time when we
add a row from each side, the index of the central square
increases by $2$ and the spanning tree congestion increases by
$4$. By induction, this implies $s(T_{3n+2})=4n+2$.

\begin{center}
\setlength{\unitlength}{2mm}
\begin{picture}(58,36)

\linethickness{0.3pt}\qbezier(10,1)(10,1)(48,1)
\linethickness{0.3pt}\qbezier(10,1)(10,1)(29,34.5)
\linethickness{0.3pt}\qbezier(29,34.5)(29,34.5)(48,1)

\linethickness{0.3pt}\qbezier(13.8,7.7)(13.8,7.7)(44.2,7.7)
\linethickness{0.3pt}\qbezier(17.6,14.4)(17.6,14.4)(40.4,14.4)
\linethickness{0.3pt}\qbezier(21.4,21.1)(21.4,21.1)(36.6,21.1)
\linethickness{0.3pt}\qbezier(25.2,27.8)(25.2,27.8)(32.8,27.8)

\linethickness{0.3pt}\qbezier(13.8,7.7)(13.8,7.7)(17.6,1)
\linethickness{0.3pt}\qbezier(17.6,14.4)(17.6,14.4)(25.2,1)
\linethickness{0.3pt}\qbezier(21.4,21.1)(21.4,21.1)(32.8,1)
\linethickness{0.3pt}\qbezier(25.2,27.8)(25.2,27.8)(40.4,1)

\linethickness{0.3pt}\qbezier(32.8,27.8)(32.8,27.8)(17.6,1)
\linethickness{0.3pt}\qbezier(36.6,21.1)(36.6,21.1)(25.2,1)
\linethickness{0.3pt}\qbezier(40.4,14.4)(40.4,14.4)(32.8,1)
\linethickness{0.3pt}\qbezier(44.2,7.7)(44.2,7.7)(40.4,1)

\put(28.1,9.9){$C$}

\linethickness{1.5pt}\qbezier(17.6,9.9)(17.6,9.9)(21.4,12.1)
\linethickness{1.5pt}\qbezier(21.4,12.1)(21.4,12.1)(25.2,9.9)
\linethickness{1.5pt}\qbezier(25.2,9.9)(25.2,9.9)(29,12.1)

\linethickness{1.5pt}\qbezier(29,12.1)(29,12.1)(29,16.6)
\linethickness{1.5pt}\qbezier(29,16.6)(29,16.6)(32.8,18.8)
\linethickness{1.5pt}\qbezier(32.8,18.8)(32.8,18.8)(32.8,23.3)

\linethickness{1.5pt}\qbezier(29,12.1)(29,12.1)(32.8,9.9)
\linethickness{1.5pt}\qbezier(32.8,9.9)(32.8,9.9)(32.8,5.4)
\linethickness{1.5pt}\qbezier(32.8,5.4)(32.8,5.4)(36.6,3.2)

\end{picture}
\bigskip

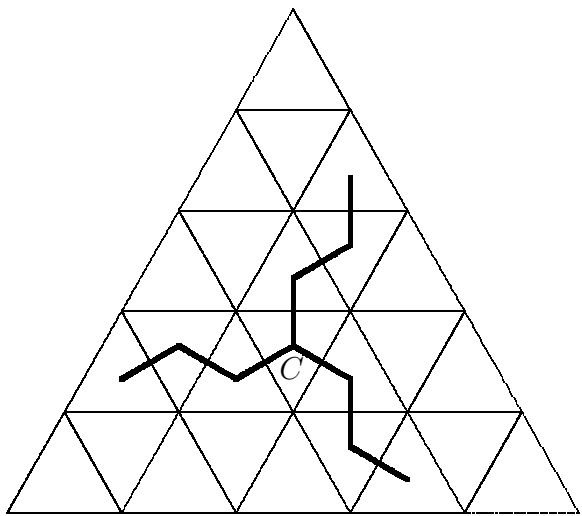
\centerline{Figure 6}
\end{center}

For $T_{3n}$ the argument is almost the same. Figure 6 shows the
suggested center-tail systems in $T_6$. The argument in this case
is a repetition of the argument for $T_{3n+2}$.
\medskip

For $T_{3n+1}$ $n\ge 1$ we suggest somewhat different center-tail
systems (because there is no central triangle). In this case the
center consists of $6$ faces (we mark all of them using $C$ on
Figure 7), but there are still three tails, and the assignment of
tails is similar to the previous cases: all edges from the bottom
side of the triangle are assigned the tail which goes in the
upward-right direction. Figure 7 shows the center and the tails
for $T_7$. The argument is quite similar to the argument for
previous cases, but now we have to compute the first minimum,
which is equal to $8$ for $T_7$ and is equal to $4n$ for
$T_{3n+1}$. The second and the third minima are equal to $10$ for
$T_7$ and to $4n+2$ for $T_{3n+1}$.
\end{proof}

\begin{center}
\setlength{\unitlength}{2mm}
\begin{picture}(58,43)

\linethickness{0.3pt}\qbezier(6.2,1)(6.2,1)(51.8,1)
\linethickness{0.3pt}\qbezier(6.2,1)(6.2,1)(29,41.2)
\linethickness{0.3pt}\qbezier(29,41.2)(29,41.2)(51.8,1)

\linethickness{0.3pt}\qbezier(10,7.7)(10,7.7)(48,7.7)
\linethickness{0.3pt}\qbezier(13.8,14.4)(13.8,14.4)(44.2,14.4)
\linethickness{0.3pt}\qbezier(17.6,21.1)(17.6,21.1)(40.4,21.1)
\linethickness{0.3pt}\qbezier(21.4,27.8)(21.4,27.8)(36.6,27.8)
\linethickness{0.3pt}\qbezier(25.2,34.5)(25.2,34.5)(32.8,34.5)

\linethickness{0.3pt}\qbezier(10,7.7)(10,7.7)(13.8,1)
\linethickness{0.3pt}\qbezier(13.8,14.4)(13.8,14.4)(21.4,1)
\linethickness{0.3pt}\qbezier(17.6,21.1)(17.6,21.1)(29,1)
\linethickness{0.3pt}\qbezier(21.4,27.8)(21.4,27.8)(36.6,1)
\linethickness{0.3pt}\qbezier(25.2,34.5)(25.2,34.5)(44.2,1)

\linethickness{0.3pt}\qbezier(32.8,34.5)(32.8,34.5)(13.8,1)
\linethickness{0.3pt}\qbezier(36.6,27.8)(36.6,27.8)(21.4,1)
\linethickness{0.3pt}\qbezier(40.4,21.1)(40.4,21.1)(29,1)
\linethickness{0.3pt}\qbezier(44.2,14.4)(44.2,14.4)(36.6,1)
\linethickness{0.3pt}\qbezier(48,7.7)(48,7.7)(44.2,1)

\put(28.1,9.9){$C$}
\put(28.5,16.6){$C$}\put(24.7,16.2){$C$}\put(32.3,16.2){$C$}\put(25.4,12.3){$C$}\put(32,12.3){$C$}

\linethickness{1.5pt}\qbezier(17.6,12.1)(17.6,12.1)(21.4,9.9)
\linethickness{1.5pt}\qbezier(21.4,9.9)(21.4,9.9)(25.2,12.1)
\linethickness{1.5pt}\qbezier(13.8,9.9)(13.8,9.9)(17.6,12.1)

\linethickness{1.5pt}\qbezier(29,18.8)(29,18.8)(29,23.3)
\linethickness{1.5pt}\qbezier(29,23.3)(29,23.3)(32.8,25.5)
\linethickness{1.5pt}\qbezier(32.8,25.5)(32.8,25.5)(32.8,30)

\linethickness{1.5pt}\qbezier(32.8,12.1)(32.8,12.1)(36.6,9.9)
\linethickness{1.5pt}\qbezier(36.6,9.9)(36.6,9.9)(36.6,5.4)
\linethickness{1.5pt}\qbezier(36.6,5.4)(36.6,5.4)(40.4,3.2)
\end{picture}
\bigskip

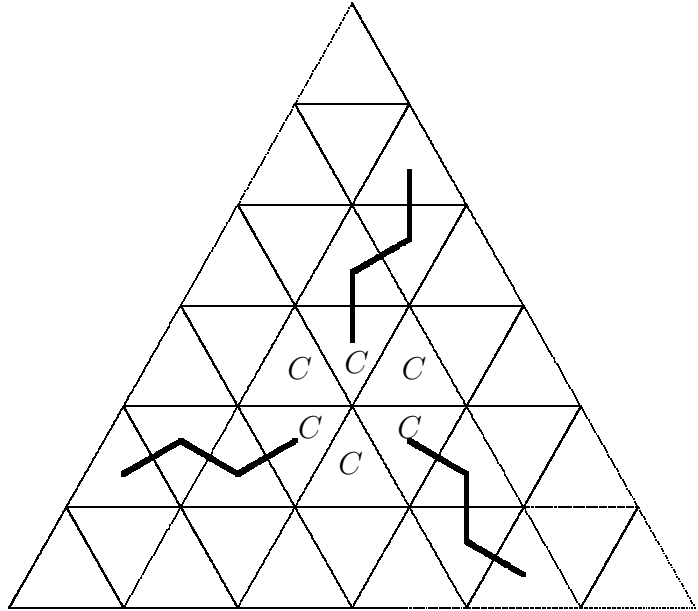
\centerline{Figure 7}
\end{center}

\begin{remark}{Remark} The values of the spanning tree congestion for the graphs
$\{T_k\}$ were studied in \cite{CCV07}. Unfortunately one of the
formulas in \cite{CCV07} is erroneous (our Theorem \ref{T:triang}
implies that the formula
$s(T_m)=2\left(\left\lfloor\frac{m-1}3\right\rfloor+\left\lfloor
\frac m3\right\rfloor\right)$ ($m\ge 4$) in \cite[Theorem
2]{CCV07} does not hold for $m=3n+2$). Also it is not clear
whether the authors of \cite{CCV07} had proofs of the
corresponding estimates from below. The reason for this doubt: the
proof of the estimate from below for square grids contains gaps
(one error is at the top of page 82: there can be several edges
connecting $M_e$ and $P$; the second error is in item (b) on page
82: one can construct examples which show that the congestion with
respect to trees with added edges is not related with the
congestion in the original tree in the stated way), and for
triangular grids no proof of the estimate from below is given in
\cite{CCV07}, the authors just say that the proof is identical
with the case of square grids.
\end{remark}

\begin{remark}{Final remark} It is not difficult to verify that
center-tail systems can be used to prove the results of
\cite{CO09} and \cite{Hru08} on rectangular planar grids and the
result of \cite[Theorem 3]{CCV07} on hexagonal grids. However, it
is far from being clear whether it is possible to use center-tail
systems to develop an algorithm for finding the spanning tree
congestion for general planar graphs.
\end{remark}

\end{large}

\begin{large}

\end{large}

\begin{thebibliography}{22}

\begin{small}

\bibitem{CH91} J. Clark and D.A.~Holton, {\it A First Look at
Graph Theory}, World Scientific, River Edge, N.J., 1991.

\bibitem{Ost04} M.\,I.~Ostrovskii, Minimal congestion trees,
{\it Discrete Math.}, {\bf 285} (2004), 219--226.

\bibitem{Car05} D.~Carr, The tree congestion of graphs, {\it preprint}, California State University, San Bernardino, 2005.

\bibitem{CCV07} A.~Castej\'on, E.~Corbacho, R.~Vidal, Optimum trees in grids
(Spanish),  Advances in discrete mathematics in Andalusia, 79--86,
Serv. Publ. Univ. C\'adiz, C\'adiz, 2007.

\bibitem{CO09} A.~Castej\'on and M.\,I.~Ostrovskii, Minimum congestion spanning trees of grids and discrete
toruses, {\it Discussiones Mathematicae Graph Theory} (2009), to
appear.

\bibitem{Cox07} A.~Cox, Tree congestion for complete $n$-partite
graphs, {\it preprint}, California State University, San
Bernardino, 2007.

\bibitem{Hru08} S.\,W.~Hruska, On tree congestion of graphs,
{\it Discrete Math.}, {\bf 308} (2008), 1801-–1809.

\bibitem{Hun07} R.~Hunter, On spanning tree congestion of product
graphs, {\it preprint}, California State University, San
Bernardino, 2007.

\bibitem{KOY09} K.~Kozawa, Y.~Otachi, K.~Yamazaki, On spanning tree congestion of
graphs, {\it Discrete Math.}, {\bf 309} (2009), 4215--4224.

\bibitem{LRR09} C.~L\"owenstein, D.~Rautenbach, F.~Regen,
On spanning tree congestion, {\it Discrete Math.}, {\bf 309}
(2009) 4653--4655.

\bibitem{Tan07} D.~Tanner, Spanning tree congestion critical
graphs, {\it preprint}, California State University, San
Bernardino, 2007.

\bibitem{Ost05} M.\,I.~Ostrovskii,
Sobolev spaces on graphs, {\it Quaestiones Mathematicae}, {\bf 28}
(2005), 501--523.

\bibitem{WC04} B.~Y.~Wu and K.-M.~Chao, {\it Spanning trees
and optimization problems}, Boca Raton, Chapman \&\ Hall/CRC,
2004.

\bibitem{BCHRS00} S.\,L.~Bezrukov, J.\,D.~Chavez, L.\,H.~Harper,
M.~R\"ottger, U.-P.~Schroeder, The congestion of $n$-cube layout
on a rectangular grid, {\it Discrete Math.}, {\bf 213} (2000),
13--19.

\bibitem{KRY94} S.~Khuller, B.~Raghavachari, N.~Young,
Designing multi-commodity flow trees, {\it Information Processing
Letters}, {\bf 50} (1994), 49--55.

\bibitem{Ros99} K. Rosen (Editor), {\it Handbook on Discrete and
Combinatorial Mathematics}, Boca Raton, CRC Press, 1999.

\bibitem{Whi32} H.~Whitney, Non-separable and planar graphs, {\it Trans. Amer.
Math. Soc.}, {\bf 34} (1932), no. 2, 339--362.

\bibitem{Whi33} H.~Whitney, Planar graphs, {\it Fundamenta Math.},
{\bf 21} (1933), 73--84.

\bibitem{Lov79} L.~Lov\'asz, {\it Combinatorial problems and exercises}, North-Holland
Publishing Co., Amsterdam-New York, 1979.

\end{small}

\end{thebibliography}
\end{document}